\documentclass[11pt]{article}
\usepackage[latin1]{inputenc}
\usepackage{amsmath,amsthm,amssymb}
\usepackage{amsfonts}
\usepackage{amsmath,amsthm,amssymb,amscd}
\usepackage{latexsym}
\usepackage{color}
\usepackage{graphicx}
\usepackage{mathrsfs}
\usepackage{cite}

\usepackage{color,enumitem,graphicx}
\usepackage[colorlinks=true,urlcolor=black,
citecolor=black,linkcolor=black,linktocpage,pdfpagelabels,
bookmarksnumbered,bookmarksopen]{hyperref}

\makeatletter \@addtoreset{equation}{section} \makeatother

\newtheorem{theorem}{Theorem}[section]

\newtheorem{lemma}{Lemma}[section]
\newtheorem{remark}{Remark}[section]

\allowdisplaybreaks

\begin{document}

\title{\bf Normalized bound state solutions for the fractional Schr\"{o}dinger equation with potential\footnote{Supported by National Natural Science Foundation of China
(No.11971393) and Natural Science Foundation of Chongqing, China (cstc2020jcyjjqX0029).}}

\author
{Xin Bao,\ Ying Lv,\ Zeng-Qi Ou\footnote{Corresponding author. E-mail address:\, {bxinss@163.com} (X. Bao), {ly0904@swu.edu.cn} (Y. Lv),
{ouzengqi@swu.edu.cn} (Z.-Q. Ou).}\\
\footnotesize  School of Mathematics and Statistics, Southwest University, Chongqing, 400715, People's Republic of China}

\date{}
\maketitle

\baselineskip 17pt

\begin{abstract}
{In this paper, we study the following fractional Schr\"{o}dinger equation with prescribed mass
\begin{equation*}
\left\{
\begin{aligned}
&(-\Delta)^{s}u=\lambda u+a(x)|u|^{p-2}u,\quad\text{in $\mathbb{R}^{N}$},\\
&\int_{\mathbb{R}^{N}}|u|^{2}dx=c^{2},\quad u\in H^{s}(\mathbb{R}^{N}),
\end{aligned}
\right.
\end{equation*}
where $0<s<1$, $N>2s$, $2+\frac{4s}{N}<p<2_{s}^{*}:=\frac{2N}{N-2s}$, $c>0$, $\lambda\in \mathbb{R}$ and $a(x)\in C^{1}(\mathbb{R}^{N},\mathbb{R}^{+})$ is a potential function. By using a minimax principle, we prove the existence of bounded state normalized solution under various conditions on $a(x)$.}

\end{abstract}

\section{Introduction and main results}
In this paper, we consider the following fractional Schr\"{o}dinger equation 
\begin{equation} \label{pr1}
\begin{aligned}
(-\Delta)^{s}u=\lambda u+a(x)|u|^{p-2}u,
\end{aligned}
\end{equation}
and possessing prescribed mass
\begin{equation}\label{pr2}
\int_{\mathbb{R}^{N}}|u|^{2}dx=c^{2},
\end{equation}
where  $0<s<1$, $N>2s$, $2+\frac{4s}{N}<p<2_{s}^{*}:=\frac{2N}{N-2s}$, $c>0$, and the function $a(x)$ satisfies the following assumptions:

$(A_{1})$ $a\in C^{1}(\mathbb{R}^{N},\mathbb{R}^{+})$, and $a_{*}:=\inf\limits_{x\in\mathbb{R}^{N}}a(x)$ with
$$
a_*>\max\Big\{\frac{N(N-2s)}{N^{2}-2s(N-2s)},\frac{2N}{Np-4s}\Big\};
$$

$(A_{2})$ $\lim\limits_{|x|\rightarrow\infty}a(x)=1$;

$(A_{3})$ $x\cdot\nabla a(x)\geqslant0$ for all $x \in \mathbb{R}^{N}$, with the strict inequality holds on a subset of positive Lebesgue

\ \ \ \ \ \ measure of $\mathbb{R}^{N}$;

$(A_{4})$ $Na(x)+x\cdot\nabla a(x)\leq N$ for all $x\in\mathbb{R}^{N}$;

$(A_{5})$ The map $W:x\mapsto x\cdot\nabla a(x)$ is in $L^{\infty}(\mathbb{R}^{N})$, and
\begin{equation}\label{W}
\|W\|_{\infty}\le \frac{(2N+p(2s-N))((Np-4s)a_*-2N)}{8(p-2)s};
\end{equation}

$(A_{6})$ $(1-a(x))\in L^{t_{1}}(\mathbb{R}^{N})$, and
\begin{equation}\label{1-a}
\|1-a(x)\|_{t_{1}}<\frac{(2^{1-\frac{4s}{N(p-2)}}-1)Np(p-2)}{(N(p-2)-4s)\|w_{c}\|_{t_{2}p}^{p}}m_{c},
\end{equation}

where
$t_{1}=\frac{2N}{2N-Np+4s}$, $t_{2}=\frac{2N}{Np-4s},$
$w_c$ and $m_c$ are defined in \eqref{lim2} and \eqref{mc} respectively.

The fractional Laplacian operator $(-\Delta)^{s}$ is defined by
$$
(-\Delta)^{s}u(x)=C_{N,s}\ P. V.\int_{\mathbb{R}^{N}}\frac{u(x)-u(y)}{|x-y|^{N+2s}}dy,\quad \text{in $\mathbb{R}^{N}$}
$$
for $u\in C_{0}^{\infty}(\mathbb{R}^{N})$, where $C_{N,s}$ is a suitable positive normalizing constant and $P.V.$ denotes the Cauchy principle value, see \cite{ED}.

When we are looking for solutions of problem \eqref{pr1}, a possible choice is to consider that $\lambda\in \mathbb{R}$ is fixed, and to look for critical points of the functional $F_{\lambda}:H^{s}(\mathbb{R}^{N})\rightarrow \mathbb{R}$ (see e.g. \cite{RC,Pd,ZFS,WZ})
\begin{equation}\label{fla}
F_{\lambda}(u):=\frac{1}{2}\int_{\mathbb{R}^{N}}|(-\Delta)^{s/2}u|^{2}dx-\frac{\lambda}{2}\int_{\mathbb{R}^{N}}|u|^{2}dx-\frac{1}{p}\int_{\mathbb{R}^{N}}a(x)|u|^{p}dx
\end{equation}
with
$$
\int_{\mathbb{R}^{N}}|(-\Delta)^{s/2}u|^{2}d x=\int_{\mathbb{R}^{N}}\int_{\mathbb{R}^{N}} \frac{|u(x)-u(y)|^{2}}{|x-y|^{N+2 s}}dxdy,
$$
where $H^{s}(\mathbb{R}^{N})$ is a Hilbert space with the inner product and norm respectively
$$
\langle u,v\rangle=\int_{\mathbb{R}^{N}}(-\Delta)^{s/2}u(-\Delta)^{s/2}vdx+\int_{\mathbb{R}^{N}} uvdx,
$$
$$
\|u\|=\Big(\int_{\mathbb{R}^{N}}|(-\Delta)^{s/2}u|^{2}dx+\int_{\mathbb{R}^{N}}|u|^{2}dx\Big)^{1/2}.
$$

Another interesting way is to search for solutions with prescribed mass, that is, \eqref{pr2} holds, and $\lambda\in\mathbb{R}$ appears as a lagrange multiplier. This type of solution is called normalized solution, and can be obtained by looking for critical points of the functional
\begin{equation}\label{f}
F(u):=\frac{1}{2}\int_{\mathbb{R}^{N}}|(-\Delta)^{s/2} u|^{2}dx-\frac{1}{p}\int_{\mathbb{R}^{N}}a(x)|u|^{p}dx
\end{equation}
on the constraint 
$$
S_{c}:=\{u\in H^{s}(\mathbb{R}^{N}):\|u\|_{2}=c\}.
$$
From the physical point of view, the normalized solution is particularly meaningful, since in addition to there is a conservation of mass, the mass has often an important physical meaning.

Considering the following Schr\"{o}dinger equation with a normalization constraint
\begin{equation}\label{pr3}
\left\{
\begin{aligned}
&(-\Delta)^s u+V(x)u=\lambda u+f(u), \quad\text{in $\mathbb{R}^{N}$}, \\
&\int_{\mathbb{R}^{N}}|u|^{2}dx=c^{2},\quad u\in H^{s}(\mathbb{R}^{N}).
\end{aligned}
\right.
\end{equation}

If $s=1$, in recent years, much attention has been paid to study problem \eqref{pr3}, see \cite{COA,TB,BB,LJ1,XFL,NS1,NS2,JCW} for $V(x)=1$ and \cite{CA,NI,RM,STC,TBR,YHD,ZWT} for $V(x)\neq $ constant. In particular, Ikoma and Miyamoto \cite{NI} studied normalized ground state solutions of problem \eqref{pr3} with the vanishing potential $V$ in the $L^{2}$-subcritical case. In the $L^{2}$-supercritical case, by using a new minimax argument, Bartsch et al. \cite{TBR} proved the existence of normalized solutions of problem \eqref{pr3} with the vanishing potential.

For $s\in(0,1)$ and $f(u)=|u|^{p-2}u$, the associated energy functional of \eqref{pr3} is given by
$$
F_{V}(u):=\frac{1}{2}\int_{\mathbb{R}^{N}}|(-\Delta)^{s/2}u|^{2}d x-\frac{1}{2}\int_{\mathbb{R}^{N}}V(x)|u|^{2}dx-\frac{1}{p}\int_{\mathbb{R}^{N}}|u|^{p}dx.
$$
From the variational point of view, $F_{V}$ is bounded from below on $S_{c}$ for $p\in (2, 2+\frac{4s}{N})$ ($L^{2}$-subcritical), and unbounded from below on $S_{c}$
 for $p\in (2+\frac{4s}{N}, 2^*_s)$ ($L^{2}$-supercritical). Here $2+\frac{4s}{N}$ is called the $L^{2}$-critical exponent, which comes from the Gagliardo-Nirenberg inequality \cite{Nirenberg1}.
There are many results about problem \eqref{pr3}, see \cite{HJL,PHZ,MDZ} for normalized solutions without potential, \cite{SBP} for normalized solutions with the vanishing potential, \cite{MD} for normalized solutions with the trapping potential, \cite{LTL} for normalized solutions with a ring-shaped potential, \cite{JBZ} for normalized solutions with a weak form of the steep well potential. 

Note that, all the results mentioned above are concerned with the problem \eqref{pr3} with autonomous nonlinearities. 
For the study of normalized solutions to the nonautonomous equations, i.e., the following equation
\begin{equation}\label{nonau}
\left\{
\begin{aligned}
&(-\Delta)^s u=\lambda u+a(x)f(u),\quad\text{in $\mathbb{R}^{N}$},\\
&\int_{\mathbb{R}^{N}}|u|^{2}dx=c^{2},\quad u\in H^{s}(\mathbb{R}^{N}).
\end{aligned}
\right.
\end{equation}
Chen and Tang \cite{STC} first considered \eqref{nonau} with $s=1$, and established the existence of normalized ground state solutions 
under suitable assumptions on $a(x)$ and $f$. Compared with the autonomous problems with a prescribed $L^2$-norm constraint, the main differences and difficulties of \cite{STC} are finding a new method to construct a Poho\v{z}aev-Palais-Smale sequence, and recover the compactness for this sequence. By using a technical condition about $a(x)$ which was introduced in \cite{DMC}, Alves \cite{C} considered the multiplicity of normalized solutions to problem \eqref{nonau} with $s=1$ in the $L^{2}$-subcritical case.
 Moreover, the above results are later generalized to the Kirchhoff type in $\mathbb{R}^3$ by Chen et al. \cite{SC} and the fractional case by Yang et al. \cite{CY}.
For more related results, see \cite{ZC,XJD,ZY} for more details.


Nowadays, the study of problem \eqref{nonau} has attracted a lot interest.  However, as far as we know, there are only a few papers dealing with problem \eqref{nonau} (even for $s=1$) besides the one already mentioned above \cite{C,SC,STC,CY}. With regard to this point, inspired by \cite{STC} and \cite{RL}, in this work, we study normalized bound state solutions of problem \eqref{pr1}-\eqref{pr2}. 

\begin{theorem}\label{th1}
Assume that $(A_{1})$-$(A_{5})$ hold, and the supremum of $|1-a(x)|$ is sufficiently small,
then problem \eqref{pr1}-\eqref{pr2} has a solution $(u,\lambda)\in H^{s}(\mathbb{R}^{N})\times \mathbb{R}^{-}$.
\end{theorem}

\begin{theorem}\label{th2}
Assume that $(A_{1})$-$(A_{6})$ hold, then problem \eqref{pr1}-\eqref{pr2} has a solution $(u,\lambda)\in H^{s}(\mathbb{R}^{N})\times \mathbb{R}^{-}$.
\end{theorem}

\begin{remark}
{\rm We point out that the supremum of $|1-a(x)|$ in Theorem \ref{th1} can be made explicit, see \eqref{a7}. }
\end{remark}


\begin{remark}
{\rm Different from \cite{STC}, where the authors considered $a(x)\geq a_\infty$ and the existence of ground state solutions, while, in this paper, we focus on the case $a(x)\leq a_\infty$, and obtain a bound state solution.}
\end{remark}

\begin{remark}
{\rm Let us give a brief illustration of our proof.

Step 1: Under the conditions of $a(x)$, we prove that $F$ has a linking geometry;

Step 2: Borrowing the idea of \cite{LJ}, and using a new minimax principle introduced in \cite{NG}, we construct a bounded Poho\v{z}aev-Palais-Smale sequence;

Step 3: By using the splitting Lemma, we prove the convengence.} 
\end{remark}

The paper is organized as follows. Section \ref{pre} contains some preliminaries. In Section \ref{minimax}, we use a minimax principle to construct a Poho\v{z}aev-Palais-Smale sequence. Section \ref{proof} is devoted to the proofs of Theorems \ref{th1} and \ref{th2}.
Throughout the paper, we will use the notation $\|\cdot\|_q:=\|\cdot\|_{L^q(\mathbb{R}^N)}$, $q\in (1,\infty)$, $B_{R}(y)=\{x \in \mathbb{R}^{N}:|x-y|\leq R\}$ is the closed ball of radius $R$ around $y$, $C,C_i,i\in \mathbb{N}^+$ denote positive constants possibly different from line to line.

\section{Preliminaries}\label{pre}


By \cite{RF}, for $p\in(2+\frac{4s}{N},2_{s}^{\ast})$, the following fractional Schr\"{o}dinger equation:
\begin{equation}\label{lim1}
(-\Delta)^{s}u=-u+u^{p-1},\quad x\in\mathbb{R}^{N},\quad u>0,
\end{equation}
has an unique solution $w\in H^{s}(\mathbb{R}^{N})$, which is radial and radially decreasing.
Then for $c>0$, one obtains a solution $(w_{c}$, $\lambda_{c})\in H_{\text {rad }}^{s}(\mathbb{R}^{N})\times \mathbb{R}^{-}$ of the following equation
\begin{equation}\label{lim2}
\left\{\begin{array}{l}
(-\Delta)^{s} w_{c}=\lambda_{c}w_{c}+|w_{c}|^{p-2}w_{c} \\
w_{c}>0,\quad\|w_{c}\|_{2}=c,
\end{array}\right.
\end{equation}
by scaling:
$$
w_{c}(x):=(-\lambda_{c})^{\frac{1}{p-2}} w((-\lambda_{c})^{\frac{1}{2s}}x),
$$
where $\lambda_{c}$ is determined by
\begin{equation}\label{lambdac}
-\lambda_{c}=\left(\frac{c}{c_{0}}\right)^{-\theta-2}>0,\quad c_{0}:=\|w\|_{2}, \quad \theta:=\frac{4N-2p(N-2s)}{N(p-2)-4s}>0,
\end{equation}
and $H_{\mathrm{rad}}^{s}(\mathbb{R}^{N})$ is the subset of the radially symmetric functions in $H^{s}(\mathbb{R}^{N})$.
The solution $w_{c}$ of \eqref{lim2} is a critical point, in fact a mountain pass critical point of
\begin{equation}\label{finfty}
F_{\infty}(u):=\frac{1}{2}\int_{\mathbb{R}^{N}}|(-\Delta)^{s/2}u|^{2}dx-\frac{1}{p}\int_{\mathbb{R}^{N}}|u|^{p}dx
\end{equation}
constrained to $S_{c}$. For any $c>0$, setting
\begin{equation}\label{mc}
m_{c}:=F_{\infty}(w_{c})\ \ \mbox{and}\ \ m_{c_{0}}=F_{\infty}(w),
\end{equation}
one has
\begin{equation}\label{mm}
\begin{aligned}
m_{c}&=\frac{1}{2}\int_{\mathbb{R}^{N}}|(-\Delta)^{s/2}w_{c}|^{2}dx-\frac{1}{p}\int_{\mathbb{R}^{N}}|w_{c}|^{p}dx\\
&=(-\lambda_{c})^{\frac{p}{p-2}-\frac{N}{2s}}\Big(\frac{1}{2}\int_{\mathbb{R}^{N}}|(-\Delta)^{s/2}w|^{2}dx-\frac{1}{p}\int_{\mathbb{R}^{N}}|w|^{p}dx\Big)\\
&=(-\lambda_{c})^{\frac{p}{p-2}-\frac{N}{2s}}m_{c_{0}}\\
&=\left(\frac{c}{c_{0}}\right)^{-\theta}m_{c_{0}}.
\end{aligned}
\end{equation}

Now we recall the notion of barycentre of a function $u \in H^{s}(\mathbb{R}^{N})\backslash\{0\}$ which has been introduced in \cite{TBT,GC}. Setting
$$
\nu(u)(x)=\frac{1}{|B_{1}(0)|}\int_{B_{1}(x)}|u(y)|dy.
$$
We can observe that $\nu(u)$ is bounded and continuous, and then, the function
$$
\widehat{u}(x)=\Big(\nu(u)(x)-\frac{1}{2} \max \nu(u)\Big)^{+}
$$
is well defined. Moreover, $\widehat{u}$ is continuous and has a compact support. Therefore, we can define $\beta: H^{s}(\mathbb{R}^{N})\backslash\{0\}\rightarrow\mathbb{R}^{N}$ as
$$
\beta(u)=\frac{1}{\|\widehat{u}\|_{1}} \int_{\mathbb{R}^{N}} \widehat{u}(x)xdx.
$$
The map $\beta$ is well defined, since $\widehat{u}$ has a compact support, and one can verify that it satisfies the following properties:

$(i)$ $\beta$ is continuous in $H^{s}(\mathbb{R}^{N})\backslash\{0\}$;

$(ii)$ if $u$ is a radial function, then $\beta(u)=0$;

$(iii)$ $\beta(tu)=\beta(u)$ for any $t\neq 0$ and $u\in H^{s}(\mathbb{R}^{N})\backslash\{0\}$;

$(iv)$ letting $u_{z}(x)=u(x-z)$ for any $z\in\mathbb{R}^{N}$ and $u\in H^{s}(\mathbb{R}^{N})\backslash\{0\}$, there holds $\beta\left(u_{z}\right)=\beta(u)+z$.


\begin{lemma}\cite[Lemma 4.5]{NG}\label{mini}
Let $M$ be a Hilbert manifold and $\mathcal{I}\in C^{1}(M,\mathbb{R})$ be a given functional. Let  $K\subset M$ be compact and consider a subset
$$
\mathcal{G}\subset\{G\subset M:G\text{ is compact, }K\subset G\}
$$
which is homotopy-stable, that is, it is invariant with respect to deformations leaving $K$ fixed. Assume that
$$
\max_{u \in K}\mathcal{I}(u)<e:=\inf_{G\in\mathcal{G}}\max_{u \in G}\mathcal{I}(u)\in \mathbb{R}.
$$
If $\tau_{n}\in \mathbb{R}$ such that $\tau_{n}\rightarrow 0$ and $G_{n}\in \mathcal{G}$ is a sequence such that
$$
0\leq\max_{u\in G_{n}}\mathcal{I}(u)-e\leq\tau_{n}.
$$
Then there exists a sequence $\{u_{n}\}\subset M$ such that

$(1)$ $|\mathcal{I}(u_{n})-e|\leq\tau_{n}$,

$(2)$ $\|\nabla_{M}\mathcal{I}(u_{n})\|\leq C\sqrt{\tau_{n}}$,

$(3)$ $\operatorname{dist}(u_{n}, G_{n})\leq C\sqrt{\tau_{n}}$,\\
for some constant $C>0$.
\end{lemma}


To study the behavior of a Palais-Smale sequence, we introduce a splitting Lemma. For $\lambda<0$, let
\begin{equation}\label{fla}
F_{\lambda}(u):=\frac{1}{2}\int_{\mathbb{R}^{N}}|(-\Delta)^{s/2}u|^{2}dx-\frac{\lambda}{2}\int_{\mathbb{R}^{N}}|u|^{2}dx-\frac{1}{p}\int_{\mathbb{R}^{N}}a(x)|u|^{p}dx
\end{equation}
and
\begin{equation}\label{finla}
F_{\infty,\lambda}(u):=\frac{1}{2}\int_{\mathbb{R}^{N}}|(-\Delta)^{s/2}u|^{2}dx-\frac{\lambda}{2}\int_{\mathbb{R}^{N}}|u|^{2}dx-\frac{1}{p}\int_{\mathbb{R}^{N}}|u|^{p}dx.
\end{equation}

\begin{lemma}\label{split}
(Splitting Lemma) Let $\{v_{n}\}\subset H^{s}(\mathbb{R}^{N})$ be a Palais-Smale sequence for  $F_{\lambda}$ such that $v_{n}\rightharpoonup v$ in $H^{s}(\mathbb{R}^{N})$. Then there exist an integer  $k\geq 0$, $k$ non-trivial solutions $u^{1}$, $u^{2}$, $\cdots$, $u^{k}\in H^{s}(\mathbb{R}^{N})$ to the limit equation
$$
(-\Delta)^{s}u=\lambda u+|u|^{p-2}u\quad\text{in $\mathbb{R}^{N}$}
$$
and $k$ sequences $\{y_{n}^{j}\} \subset \mathbb{R}^{N}$, $1 \leq j \leq k$, such that  $|y_{n}^{j}| \rightarrow \infty$ as $n \rightarrow \infty$, and
\begin{equation}\label{s1}
v_{n}=v+\sum_{j=1}^{k}u^{j}(\cdot-y_{n}^{j})+o(1) \quad \text {in $H^{s}(\mathbb{R}^{N})$}.
\end{equation}
Moreover, we have
\begin{equation}\label{s2}
\|v_{n}\|_{2}^{2}=\|v\|_{2}^{2}+\sum_{j=1}^{k}\|u^{j}\|_{2}^{2}+o(1)
\end{equation}
and
\begin{equation}\label{s3}
F_{\lambda}(v_{n})\rightarrow F_{\lambda}(v)+\sum_{j=1}^{k} F_{\infty,\lambda}(u^{j})\ \ \text{as }n\rightarrow \infty.
\end{equation}
\end{lemma}
\begin{proof}
The proof of Lemma \ref{split} can be found in \cite[Lemma 3.1]{JC}. The only difference is that \cite{JC} deals with exterior domains, not with in the whole space, however the proof is exactly the same with $\lambda<0$ and thus we omit it here.
\end{proof}

\section{The minimax approach}\label{minimax}

For $h\in\mathbb{R}$ and $u\in H^{s}(\mathbb{R}^{N})$, we introduce the scaling
$$
h\star u(x):=e^{\frac{Nh}{2}}u(e^{h}x),
$$
which preserves the $L^{2}$-norm: $\|h \star u\|_{2}=\|u\|_{2}$ for all $h\in\mathbb{R}$. For $R>0$ and $h_{1}<0<   h_{2}$, which will be determined later, we set
$$
Q:=B_{R}(0)\times[h_{1}, h_{2}]\subset \mathbb{R}^{N} \times \mathbb{R}.
$$
 For $c>0$ we define
$$
\Gamma_{c}:=\left\{\gamma\in C(Q, S_{c}): \gamma(y,h)=h\star w_{c}(\cdot-y) \text { for all }(y,h)\in\partial Q\right\}.
$$
We want to find a solution of problem \eqref{pr1}-\eqref{pr2} whose energy is given by
\begin{equation}\label{mvc}
m_{a,c}:=\inf_{\gamma\in\Gamma_{c}}\max_{(y,h)\in Q}F(\gamma(y,h)).
\end{equation}
In order to develop a min-max argument, we need to prove that
$$
\sup_{\gamma\in\Gamma_{c}}\max_{(y,h)\in\partial Q}F(\gamma(y, h))<m_{a,c}
$$
at least for some suitable choice of $Q$. For this purpose, we will prove Lemma \ref{inf}, which gives a lower bound of $m_{a,c}$, and Lemma \ref{rh}, which gives an upper bound of $F\circ\gamma$ on the boundary $\partial Q$, for any given $\gamma \in \Gamma_{c}$. The values of $R>0$ and $h_{1}<0<h_{2}$ will be determined in Lemma \ref{rh}. 

Next, for $h_{1}<0<h_{2}$ determined in Lemma \ref{rh}, we define
$$
\begin{aligned}
&\mathcal{D}:=\{D\subset S_{c}:D\text{ is compact and connected, }h_{1}\star w_{c},\ \  h_{2}\star w_{c}\in D\},\\
&\mathcal{D}_{0}:=\{D\in\mathcal{D}:\beta(u)=0\text{ for all } u\in D\},\\
&\mathcal{D}_{r}:=\mathcal{D} \cap H_{\mathrm{rad}}^{s}(\mathbb{R}^{N}),
\end{aligned}
$$
and
$$
\begin{aligned}
l_{c}&:=\inf_{D\in\mathcal{D}}\max_{u \in D}F_{\infty}(u),\\
l_{c}^{0}&:=\inf_{D\in\mathcal{D}_{0}}\max_{u \in D}F_{\infty}(u),\\
l_{c}^{r}&:=\inf_{D\in\mathcal{D}_{r}}\max_{u \in D}F_{\infty}(u).
\end{aligned}
$$
By a similar argument as \cite{HJL}, we obtain
$$
m_{c}=\inf_{\sigma\in\Sigma_{c}}\max_{t\in(0,1)}F_{\infty}(\sigma(t)),
$$
where
$$
\Sigma_{c}=\{\sigma\in C([0,1],S_{c}):\sigma(0)=h_{1}\star w_{c},\ \ \sigma(1)=h_{2}\star w_{c}\}.
$$
\begin{lemma}\label{equi}
$l_{c}^{r}=l_{c}^{0}=l_{c}=m_{c}$.
\end{lemma}
\begin{proof}
Since $\mathcal{D}_{r}\subset\mathcal{D}_{0}\subset \mathcal{D}$, we have $l_{c}^{r}\geq l_{c}^{0}\geq l_{c}$. In order to prove that $l_{c}\geq m_{c}$ and $m_{c}\geq l_{c}^{r}$, arguing by contradiction we assume that $l_{c}<m_{c}$. Then $\max\limits_{u \in D}F_{\infty}(u)<m_{c}$ for some $D\in \mathcal{D}$, hence $\max\limits_{u\in U_{\delta}(D)}F_{\infty}(u)<m_{c}$ for some $\delta>0$, where $U_{\delta}(D)$ is a $\delta$-neighborhood of $D$. Observe that $U_{\delta}(D)$ is open and connected, so it is path-connected. Therefore there exists a path $\sigma\in\Sigma_{c}$ such that $\max\limits_{t \in(0,1)}F_{\infty}(\sigma(t))<m_{c}$, which is a contradiction.

On the other hand, let $\tilde{D}:=\{h\star w_{c}:h \in[h_{1}, h_{2}]\}\in\mathcal{D}_{r}$, then
$$
\max_{u\in\tilde{D}}F_{\infty}(u)=\max_{h\in[h_{1},h_{2}]}F_{\infty}(h\star w_{c})=F_{\infty}(w_{c})=m_{c}.
$$
Hence, $m_{c}\geq l_{c}^{r}$.
\end{proof}

\begin{lemma}\label{lm}
$L_{c}:=\inf\limits_{D\in\mathcal{D}_{\mathrm{0}}}\max\limits_{u\in D}F(u)>m_{c}$.
\end{lemma}
\begin{proof}
By $(A_{3})$, $(A_{4})$ and Lemma \ref{equi}, we know
$$
\max_{u\in D}F(u)\geq \max_{u \in D}F_{\infty}(u)\geq l_{c}^{0}=m_{c}\quad\text {for all $D\in\mathcal{D}_{0}$}.
$$
Now we argue by contradiction and assume that there exists a sequence $\{D_n\}\subset\mathcal{D}_{0}$ such that
$$
\max_{u \in D_{n}}F(u)\rightarrow m_{c}.
$$
Then we have
\begin{equation}\label{l1}
\max_{u\in D_{n}}F_{\infty}(u)\rightarrow m_{c}.
\end{equation}
Borrowing the idea of \cite[Lemma 2.4]{LJ}, we consider the functional
$$
\tilde{F}_{\infty}: H^{s}(\mathbb{R}^{N}) \times \mathbb{R} \rightarrow \mathbb{R}, \quad \tilde{F}_{\infty}(u, h):=F_{\infty}(h \star u)
$$
constrained to $M:=S_{c} \times \mathbb{R}$. We apply Lemma \ref{mini} with
$$
K:=\{(h_{1} \star w_{c}, 0),(h_{2} \star w_{c}, 0)\}
$$
and
$$
\mathcal{G}:=\{G \subset M: G \text { is compact, connected, } K\subset G\}.
$$
Observe that
$$
\tilde{l}_{c}:=\inf_{G \in \mathcal{G}} \max_{(u, h) \in G} \tilde{F}_{\infty}(u, h)=l_{c}=m_{c}.
$$
In fact, since $\mathcal{D} \times\{0\} \subset \mathcal{G}$, hence $l_{c} \geq \tilde{l}_{c}$, and for any $G \in \mathcal{G}$, we have $D:=\{h\star u:(u,h)\in G\} \in \mathcal{D}$ and
$$
\max_{(u,h)\in G}\tilde{F}_{\infty}(u,h)=\max_{(u, h)\in G}F_{\infty}(h\star u)=\max_{\bar{u}\in D} F_{\infty}(\bar{u}),
$$
hence $l_{c} \leq \tilde{l}_{c}$. Therefore, Lemma \ref{mini} yields a sequence $(u_{n}, h_{n}) \in S_{c} \times \mathbb{R}$ such that

$(1)$  $|\tilde{F}_{\infty}(u_{n}, h_{n})-m_{c}|\rightarrow 0$ as $n\rightarrow \infty$,

$(2)$  $\|\nabla_{S_{c}\times\mathbb{R}}\tilde{F}_{\infty}(u_{n}, h_{n})\| \rightarrow 0$ as $n \rightarrow \infty $,

$(3)$  $\operatorname{dist}((u_{n}, h_{n}), D_{n} \times\{0\}) \rightarrow 0$ as $n \rightarrow \infty$.

In particular, let $v_{n}:=h_{n}\star u_{n}$, differentiation shows that$\{v_{n}\}\subset S_{c}$ is a Palais-Smale sequence for $F_{\infty}$ on $S_{c}$ at level $m_{c}$ satisfying the Poho\v{z}aev identity for $F_{\infty}$, that is there exist Lagrange multipliers $\mu_{n}\in\mathbb{R}$ such that
$$
\begin{aligned}
&\frac{1}{2}\|(-\Delta)^{s/2}v_{n}\|_{2}^{2}-\frac{1}{p}\|v_{n}\|_{p}^{p}\rightarrow m_{c}, \\
&\|F_{\infty}^{\prime}(v_{n})-\mu_{n}H^{\prime}(v_{n})\|_{(H^{s}(\mathbb{R}^{N}))^*}\rightarrow 0,\quad\text{ where }H(u)=\frac{1}{2}\int_{\mathbb{R}^{N}}|u|^{2}dx,\\
&s\|(-\Delta)^{s/2}v_{n}\|_{2}^{2}-\frac{N(p-2)}{2p}\|v_{n}\|_{p}^{p}\rightarrow 0,\\
\end{aligned}
$$
as $n\rightarrow\infty$. Moreover, $\{v_{n}\}$ is bounded in $H^{s}(\mathbb{R}^{N})$. In fact, from the first and third relations above, we can infer that
$$
\|(-\Delta)^{s/2}v_{n}\|_{2}^{2}\rightarrow\frac{2N(p-2)}{N(p-2)-4s}m_{c}
$$
and
$$
\|v_{n}\|_{p}^{p}\rightarrow\frac{4ps}{N(p-2)-4s}m_{c}.
$$
Thus $\{v_{n}\}$ is bounded in $H^{s}(\mathbb{R}^{N})$, and the second relation implies that
$$
\|(-\Delta)^{s/2}v_{n}\|_{2}^{2}-\mu_{n}\|v_{n}\|_{2}^{2}-\|v_{n}\|_{p}^{p}\rightarrow 0,
$$
which is equivalent to
$$
\mu_{n}=\frac{\|(-\Delta)^{s/2}v_{n}\|_{2}^{2}-\|v_{n}\|_{p}^{p}}{\|v_{n}\|_{2}^{2}}+o(1)
\rightarrow\frac{N(p-2)-2ps}{N(p-2)-4s}\cdot\frac{2m_{c}}{c^{2}}=:\mu<0.
$$
Therefore, after passing to subsequences, $\{v_{n}\}$ converges weakly in $H^{s}(\mathbb{R}^{N})$ to a solution $v \in H^{s}(\mathbb{R}^{N})$ of $(-\Delta)^{s}v=\mu v+|v|^{p-2}v$. We claim that
\begin{equation}\label{l2}
v_{n} \rightarrow \pm w_{c}\quad\text{strongly in $H^{s}(\mathbb{R}^{N})$}.
\end{equation}
In order to see this we first observe that, using again the second relation above,
$$
\int_{\mathbb{R}^{N}}(-\Delta)^{s/2}v_{n}(-\Delta)^{s/2}\varphi dx-\mu_{n}\int_{\mathbb{R}^{N}}v_{n}\varphi dx-\int_{\mathbb{R}^{N}}|v_{n}|^{p-2}v_{n}\varphi dx=o(1)\|\varphi\|
$$
for every $\varphi \in H^{s}(\mathbb{R}^{N})$, hence $\{v_{n}\}$ is also a Palais-Smale sequence for $F_{\infty, \mu}$ at the level $m_{c}-\frac{\mu}{2}c^{2}$. As a consequence, Lemma \ref{split} implies
$$
v_{n}=v+\sum_{j=1}^{m} u^{j}(\cdot-y_{n}^{j})+o(1)
$$
in $H^{s}(\mathbb{R}^{N})$, where $m\geq 0$ and $u^{j}\neq 0$ are solutions to
$$
(-\Delta)^{s}u^{j}=\mu u^{j}+|u^{j}|^{p-2}u^{j}
$$
and $|y_{n}^{j}| \rightarrow \infty$. Moreover, setting $\gamma:=\|v\|_{2}$ and  $\alpha_{j}:=\|u^{j}\|_{2}$, then there holds
$$
c^{2}=\gamma^{2}+\sum_{j=1}^{m} \alpha_{j}^{2}
$$
and thus at least one of the limit functions must be non-trivial. In addition, we have
$$
F_{\infty}(v_{n})-\frac{\mu}{2}c^{2}=F_{\infty}(v)-\frac{\mu}{2} \gamma^{2}+\sum_{j=1}^{m}(F_{\infty}(u^{j})-\frac{\mu}{2}\alpha_{j}^{2})+o(1),
$$
which yields that
$$
F_{\infty}(v_{n})=F_{\infty}(v)+\sum_{j=1}^{m} F_{\infty}(u^{j})+o(1) .
$$
Using Appendix in \cite{SBP} and \eqref{mm}, we have that, if $v$ is non-trivial, then $F_{\infty}(v) \geq m_{\gamma}\geq m_{c}$. Therefore, if $m=0$, then $v\neq0$; if $m=1$, since $u^{j}\neq0$, we have $F_{\infty}(u^{j})\geq m_{\alpha_{j}}\geq m_{c}$, then
$$
m_{c}+o(1)=F_{\infty}(v_{n})=F_{\infty}(v)+\sum_{j=1}^{m}F_{\infty}(u^{j})+o(1)\geq 2m_{c}+o(1),
$$
which implies that $v=0$.

If $m=1$ and $v=0$, then $v_{n}=u^{1}(x-y_{n}^{1})+o(1)$, that is, $\tilde{v}_{n}:=v_{n}(\cdot+y_{n}^{1})=u^{1}+o(1)$. On the other hand, due to point $(3)$, we have that $v_{n}=\varphi_{n}+o(1)$ with $\varphi_{n} \in D_{n}$, so that  $\beta(\varphi_{n})=0$. Moreover, by the continuity of $\beta$ in $H^{s}(\mathbb{R}^{N})$, $u^{1}=\tilde{v}_{n}+o(1)$ and $\tilde v_{n}=\varphi_{n}(\cdot+y_{n}^{1})+o(1)$, we obtain
$$
\beta(u^{1})=\beta(\tilde{v}_{n})+o(1)=\beta(\varphi_{n}(\cdot+y_{n}^{1}))+o(1)=y_{n}^{1}+o(1),
$$
which contradicts $|y_{n}^{1}|\rightarrow \infty$.
So far we deduced that $m=0$ and $v_{n} \rightarrow v$ strongly in $H^{s}(\mathbb{R}^{N})$. Using again point $(3)$, 
we have $\beta(v)=0$. Hence, by uniqueness, \eqref{l2} follows. This implies
$$
F(v_{n})=F_{\infty}(v_{n})+\frac{1}{p} \int_{\mathbb{R}^{N}}(1-a(x))|v_{n}|^{p}dx\rightarrow m_{c}+\frac{1}{p} \int_{\mathbb{R}^{N}}(1-a(x))|w_{c}|^{p}dx>m_{c}
$$
as $n\rightarrow \infty$, which contradicts \eqref{l1}.
\end{proof}
\begin{lemma}\label{inf}
For any $c>0$ there holds $m_{a,c}\geq L_{c}$.
\end{lemma}
\begin{proof}
Given a function $\gamma: Q=B_{R}(0)\times[h_{1},h_{2}]\rightarrow S_{c}$ and  $h\in[h_{1}, h_{2}]$ we consider the mapping
$$
f_{h}:B_{R}(0)\rightarrow \mathbb{R}^{N},\quad y\mapsto\beta\circ\gamma(y, h) .
$$
We note that $f_{h_{i}}(y)=0$ if and only if $y=0$ for $i=1,2$, and $f_{h}(y)=y \neq 0$ for any $y \in   \partial B_{R}(0)$, so that $\operatorname{deg}(f_{h},B_{R}(0),0)=1$ for all $h \in[h_{1}, h_{2}]$. Therefore, by degree theory, there exists a connected compact set $Q_{0} \subset Q$ such that $(0, h_{i})\in Q_{0}$ for $i=1 , 2$, and $\beta \circ \gamma(y, h)=0$ for any $(y, h) \in Q_{0}$. Hence the set $D_{0}:=\gamma(Q_{0})\in\mathcal{D}_{0}$ satisfies
$$
\max _{(y, h) \in Q} F(\gamma(y, h)) \geq \max _{u \in D_{0}}F(u),
$$
which concludes the proof.
\end{proof}
\begin{lemma}\label{rh}
For any $c>0$ and $\varepsilon>0$ there exist $\bar{R}>0$ and  $\bar{h}_{1}<0<\bar{h}_{2}$ such that for $Q=B_{R}(0)\times[h_{1}, h_{2}]$ with $R \geq \bar{R}$, $h_{1} \leq \bar{h}_{1}$, $h_{2} \geq \bar{h}_{2}$ the following holds:
\begin{equation}\label{r1}
\max_{(y,h)\in\partial Q}F(h\star w_{c}(\cdot-y))<m_{c}+\varepsilon.
\end{equation}
\end{lemma}
\begin{proof}
Since
$$
F(h\star w_{c}(\cdot-y))=\frac{e^{2sh}}{2}\int_{\mathbb{R}^{N}}|(-\Delta)^{s/2}w_{c}|^{2} dx-\frac{e^{\frac{p-2}{2}Nh}}{p}\int_{\mathbb{R}^{N}}a(e^{-h}x+y)|w_{c}|^{p}dx,
$$
from $a_*\leq a(x)\leq 1$ for any $x\in \mathbb{R}^{N}$ and $p>2+\frac{4s}{N}$,  we deduce that
$$
F(h\star w_{c}(\cdot-y))\rightarrow-\infty,\quad\text{as $h\rightarrow \infty$}
$$
uniformly in $y\in B_{R}(0)$ for any $R>0$, and
$$
F(h\star w_{c}(\cdot-y))\rightarrow0,\quad\text{as $h\rightarrow -\infty$}
$$
uniformly in $y\in\mathbb{R}^{N}$. Thus,
\begin{equation}\label{r2}
\max_{y\in B_{R}(0),h\in\{h_{1},h_2\}}F(h\star w_{c}(\cdot-y))<m_{c},
\end{equation}
provided that $h_{1}<0$ is small enough and $h_{2}>0$ is large enough. Moreover, for $|y|=R$ large enough and $h \in[h_{1}, h_{2}]$, choosing $\alpha \in(0,1)$ such that $\alpha(1+e^{-h_{1}})<1$, then we have
\begin{equation}\label{r3}
\begin{aligned}
&\int_{\mathbb{R}^{N}}(1-a(x))|(h\star w_{c})(x-y)|^{p}dx \\
\leq &\int_{|x|>\alpha R}(1-a(x))|(h\star w_{c})(x-y)|^{p}dx+\int_{|x-y|>\alpha Re^{-h}}(1-a(x))|(h\star w_{c})(x-y)|^{p}dx.
\end{aligned}
\end{equation}
The first integral is bounded by
\begin{equation}\label{r4}
\int_{|x|>\alpha R}(1-a(x))|(h\star w_{c})(x-y)|^{p}dx\le e^{\frac{p-2}{2}Nh}\|1-a(x)\|_{L^{\infty}(|x|>\alpha R)}\|w_{c} \|_{p}^{p}\rightarrow0
\end{equation}
as $R \rightarrow \infty$ and
\begin{equation}\label{r5}
\begin{aligned}
&\int_{|x-y|>\alpha Re^{-h}}(1-a(x))|(h\star w_{c})(x-y)|^{p}dx\\
=&e^{\frac{p}{2}Nh}\int_{|x-y|>\alpha Re^{-h}}(1-a(x))|w_{c}(e^{h}(x-y))|^{p}dx\\
=&e^{\frac{p-2}{2}Nh}\int_{|\xi|>\alpha R}(1-a(e^{-h}\xi+y))|w_{c}|^{p}d\xi\\
\leq&e^{\frac{p-2}{2}Nh}\|1-a(e^{-h}\xi+y)\|_\infty\|w_{c}\|_{L^{p}(|\xi|>\alpha R)}^{p}\rightarrow 0
\end{aligned}
\end{equation}
as $R \rightarrow \infty$.
Note that
$$
F(h \star w_{c}(\cdot-y))=F_{\infty}(h \star w_{c})+\frac{1}{p} \int_{\mathbb{R}^{N}}(1-a(x))|(h \star w_{c})(x-y)|^{p}dx,
$$
form \eqref{r3}, \eqref{r4} and \eqref{r5}, it follows that
\begin{equation}\label{r6}
\max_{|y|=R, h\in[h_{1},h_2]}F(h\star w_{c}(\cdot-y))<m_{c}+\varepsilon,
\end{equation}
as $R \rightarrow \infty$.
Combining \eqref{r2} with \eqref{r6}, \eqref{r1} is proved.
\end{proof}

In the following, we always assume $|h_{1}|$, $h_{2}$ and $R$ large enough but fixed. \eqref{r1} implies that $F$ has a linking geometry and there exists a Palais-Smale sequence of $F$ at level $m_{a,c}$. The aim of the next parts is to prove that $m_{a,c}$ is a critical value for $F$. Moreover, for future purposes, we prove the following lemmas.

\begin{lemma}\label{sup}
Assume that
\begin{equation}\label{a7}
\sup_{x\in\mathbb{R}^{N}}|1-a(x)|<\frac{pm_{c}}{e^{\frac{p-2}{2}Nh_{2}}\|w_{c}\|_{p}^{p}},
\end{equation}
where $h_{2}$ is given by Lemma \ref{rh}, then $m_{a,c}<2m_{c}$.
\end{lemma}
\begin{proof}
This follows from
$$
\begin{aligned}
m_{a,c}&\leq \max_{(y,h)\in Q}\Big\{F_{\infty}(h \star w_{c}(\cdot-y))+\frac{1}{p} \int_{\mathbb{R}^{N}}(1-a(x))|(h \star w_{c})(x-y)|^{p}dx\Big\} \\
&\leq m_{c}+\frac{1}{p}e^{\frac{p-2}{2}Nh_{2}}\sup_{x\in\mathbb{R}^{N}}|1-a(x)|\cdot\|w_{c}\|_{p}^{p}\\
&< 2m_{c},
\end{aligned}
$$
provided that $|h_{1}|$, $h_{2}$ are large enough.
\end{proof}

Now we will construct a bounded Palais-Smale sequence $\{v_n\}\subset H^{s}(\mathbb{R}^{N})$ of $F$ at level $m_{a,c}$. Adapting the approach of \cite{LJ}, we introduce the following $C^{1}$-functional
$$
\tilde{F}(u,h):=F(h\star u)\ \text {for all $(u,h)\in H^{s}(\mathbb{R}^{N})\times\mathbb{R}$},
$$
and define
$$
\tilde{\Gamma}_{c}:=\{\tilde{\gamma}\in C(Q,S_{c}\times \mathbb{R}):\tilde{\gamma}(y,h):=(h \star w_{c}(\cdot-y),0)\text{ for all }(y, h) \in \partial Q\}
$$
and
$$
\tilde{m}_{a,c}:=\inf_{\tilde{\gamma}\in \tilde{\Gamma}_{c}}\max_{(y, h)\in Q}\tilde{F}(\tilde{\gamma}(y,h)).
$$
\begin{lemma}
$\tilde{m}_{a,c}=m_{a,c}$.
\end{lemma}
\begin{proof}Since $\Gamma_{c}\times\{0\}\subset \tilde{\Gamma}_{c}$, then $m_{a,c} \geq \tilde{m}_{a,c}$. On the other hand, for any $\tilde{\gamma}=(u,h)\in \tilde{\Gamma}_{c}$, the function $\gamma:=h\star u\in \Gamma_{c}$ satisfies
$$
\max _{(y, h) \in Q} \tilde{F}(\tilde{\gamma}(y, h))=\max _{(y, h) \in Q} F(\gamma(y, h)),
$$
so that $m_{a,c}\leq \tilde{m}_{a,c}$.
\end{proof}
\begin{lemma}\label{fps}
If $\{(u_{n}, h_{n})\}$ is a $(PS)_{c'}$ sequence for $\tilde{F}$ and $h_{n} \rightarrow 0$, then  $\{(h_{n} \star u_{n})\}$ is a $(PS)_{c'}$ sequence for $F$.
\end{lemma}
\begin{proof}
Let $w_{n}:=h_{n}\star u_{n}$, we have $F(w_{n})=F(h_{n}\star u_{n})=\tilde{F}(u_{n},h_{n})\rightarrow c'$. We claim that $F'(w_{n})\rightarrow 0$ in $(H^{s}(\mathbb{R}^{N}))^*$.

First of all, for any $\tilde{\varphi}\in T_{u_{n}}:=\{\tilde{\varphi}\in H^{s}(\mathbb{R}^{N}):\int_{\mathbb{R}^{N}}u_{n}\tilde{\varphi}dx=0\}$, one has
$$
\langle\partial_{u}\tilde{F}(u_{n},h_{n}),\tilde{\varphi}\rangle\rightarrow0\ \ \mbox{as}\ n\rightarrow\infty.
$$

For any $\varphi\in T_{w_{n}}:=\{\varphi\in H^{s}(\mathbb{R}^{N}):\int_{\mathbb{R}^{N}}w_{n}\varphi dx=0\}$, let $\tilde{\varphi}(x)=e^{-\frac{N}{2}h_{n}}\varphi(e^{-h_n}x)$,
we have
$$
\int_{\mathbb{R}^{N}}u_{n}\tilde{\varphi}dx
=e^{-\frac{Nh_{n}}{2}}\int_{\mathbb{R}^{N}}u_{n}(x)\varphi(e^{-h_{n}}x)dx=
\int_{\mathbb{R}^{N}}w_{n}\varphi dx=0,
$$
that is, $\tilde{\varphi}\in T_{u_{n}}$. Hence, from $h_{n} \rightarrow 0$, we have
\begin{equation*}
\begin{aligned}
&\langle F'(w_{n}),\varphi\rangle
\\=&\int_{\mathbb{R}^{N}}(-\Delta)^{s/2}w_{n}(-\Delta)^{s/2}\varphi dx-\int_{\mathbb{R}^{N}}a(x)|w_{n}|^{p-2}w_{n}\varphi dx\\
=&e^{-\frac{N}{2}h_{n}+2sh_{n}}\int_{\mathbb{R}^{N}}(-\Delta)^{s/2}u_{n}(-\Delta)^{s/2}\varphi(e^{-h_{n}}x)dx
-e^{\frac{p-3}{2}Nh_{n}}\int_{\mathbb{R}^{N}}a(e^{-h_{n}}x)|u_{n}|^{p-2}u_{n}\varphi(e^{-h_{n}}x)dx,
\\
=&e^{2sh_{n}}\int_{\mathbb{R}^{N}}(-\Delta)^{s/2}u_{n}(-\Delta)^{s/2}\tilde{\varphi}dx
-e^{\frac{p-2}{2}Nh_{n}}\int_{\mathbb{R}^{N}}a(e^{-h_{n}}x)|u_{n}|^{p-2}u_{n}\tilde{\varphi}dx\\
=&\langle\partial_{u}\tilde{F}(u_{n},h_{n}),\tilde{\varphi}\rangle\rightarrow 0,
\end{aligned}
\end{equation*}
as $n\rightarrow\infty$. the conclusion is proved.
\end{proof}

Applying Lemma \ref{mini}, we immediately have the following proposition.
\begin{lemma}\label{g}
Let $\tilde{g}_{n} \in \tilde{\Gamma}_{c}$ be a sequence such that
$$
\max_{(y, h) \in Q} \tilde{F}(\tilde{g}_{n}(y, h)) \leq m_{a,c}+\frac{1}{n} .
$$
Then there exist a sequence $(u_{n}, h_{n})\in S_{c} \times \mathbb{R}$ and $C>0$ such that

\begin{equation}\label{g1}
m_{a,c}-\frac{1}{n} \leq \tilde{F}(u_{n}, h_{n})\leq m_{a,c}+\frac{1}{n},
\end{equation}
\begin{equation}\label{g2}
\min_{(y, h)\in Q}\|(u_{n},h_{n})-\tilde{g}_{n}(y,h)\|_{H^{s}(\mathbb{R}^{N})\times \mathbb{R}}\leq\frac{C}{\sqrt{n}},
\end{equation}
\begin{equation}\label{g3}
\|\nabla_{S_{c}\times\mathbb{R}}\tilde{F}(u_{n},h_{n})\|\leq\frac{C}{\sqrt{n}}.
\end{equation}

The last inequality means:
$$
|D\tilde{F}(u_{n},h_{n})((z,s))|\leq \frac{C}{\sqrt{n}}(\|z\|+|s|)
$$
for all
$$
(z,s)\in\Big\{(z,s)\in H^{s}(\mathbb{R}^{N})\times \mathbb{R}:\int_{\mathbb{R}^{N}}zu_{n}dx=0\Big\}.
$$
\end{lemma}

\begin{lemma}\label{psbound}
There exists a bounded sequence $\{v_{n}\}$ in $S_{c}$ such that
\begin{equation}\label{p1}
 F(v_{n})\rightarrow m_{a,c},\quad \nabla_{S_{c}}F(v_{n})\rightarrow0
\end{equation}
and
\begin{equation}\label{p2}
s\|(-\Delta)^{s/2}v_{n}\|_{2}^{2}-\frac{N(p-2)}{2p}\int_{\mathbb{R}^{N}}a(x)|v_{n}|^{p}dx+\frac{1}{p}\int_{\mathbb{R}^{N}}(x\cdot\nabla a(x))|v_{n}|^{p}dx\rightarrow 0
\end{equation}
as $n\rightarrow \infty$. Moreover, the sequence of of Lagrange multipliers
\begin{equation}\label{p3}
\lambda_{n}:=\frac{\langle F'(v_{n}),v_{n}\rangle}{c^{2}}
\end{equation}
admits a subsequence $\lambda_{n}\rightarrow\lambda$ with
\begin{equation}\label{p4}
-\frac{\delta_{0}}{c^{2}}<\lambda<0,
\end{equation}
where
\begin{equation}\label{p5}
 \delta_{0}:=\Big(\frac{4(p(2s-N)+2N)}{N(p-2)-4s}+\frac{(p-2)\|w\|_{\infty}}{N(p-2)-4s}\cdot\frac{16s}{(Np-4s)a_{\ast}-2N}\Big)m_{c}.
\end{equation}
\end{lemma}
\begin{proof}
By the definition of $m_{a,c}$, we choose a sequence $g_{n}\in\Gamma_{c}$ such that
$$
\max_{(y,h)\in Q}F(g_{n}(y,h))\le m_{a,c} +\frac{1}{n}.
$$
Since $F(u)\geq F(|u|)$ for any $u\in H^{s}(\mathbb{R}^{N})$ and $|\gamma|\in \Gamma_c$ for any $\gamma\in \Gamma_c$, we can assume that $g_{n}(y,h)\geq0$ almost everywhere in $\mathbb{R}^{N}$. Define $\tilde{g_{n}}(y,h):=(g_{n}(y,h),0)$, then $\tilde{g_{n}}(y,h)\in\tilde{\Gamma}_{c}$, and
$$
\max_{(y,h)\in Q}\tilde{F}(\tilde{g_{n}}(y,h))=\max_{(y,h)\in Q}\tilde{F}((g_{n}(y,h),0)=\max_{(y,h)\in Q}{F}({g_{n}}(y,h))\le m_{a,c} +\frac{1}{n}.
$$
Applying Lemma \ref{g}, we can prove the existence of a sequence $\{(u_{n}, h_{n})\}\subset H^{s}(\mathbb{R}^{N})\times\mathbb{R}$ such that $F(h_{n}\star u_{n})\rightarrow m_{a,c}$. We also note that
$$
\min_{(y, h)\in Q}\|(u_{n}, h_{n})-\tilde{g}_{n}(y, h)\|_{H^{s}(\mathbb{R}^{N})\times\mathbb{R}}\leq\frac{C}{\sqrt{n}},
$$
so that $h_{n}\rightarrow 0$ as $n\rightarrow\infty$ and there exists $(y_{n},\bar{h}_{n})\in B_{R}(0)\times[h_{1}, h_{2}]$ such that $u_{n}-g_{n}(y_{n},\bar{h}_{n})\rightarrow0$ in $H^{s}(\mathbb{R}^N)$ as $n\rightarrow+\infty$.
We define
$$
v_{n}:=h_{n}\star u_{n}.
$$
Observe that, since $g_{n}(y_{n},\bar{h}_{n})\geq 0$ a.e. in $\mathbb{R}^{N}$, then $\|u_{n}^{-}\|_{2} \leq\|u_{n}-g_{n}(y_{n},\bar{h}_{n})\|_{2}=o(1)$ and we can deduce that $u_{n}^{-} \rightarrow 0$ a.e. in $\mathbb{R}^N$. So
$$
\|v_{n}^{-}\|_{2} \rightarrow 0\quad \text {as } n\rightarrow\infty.
$$
Moreover, by Lemma \ref{fps}, $\{v_{n}\}$ is a Palais-Smale sequence for $F$, that is
\begin{equation}\label{p6}
\frac{1}{2}\|(-\Delta)^{s/2}v_{n}\|_{2}^{2}-\frac{1}{p}\int_{\mathbb{R}^{N}}a(x)|v_{n}|^{p}dx\rightarrow m_{a,c}.
\end{equation}
and
\begin{equation}\label{p7}
\|F^{\prime}(v_{n})-\lambda_{n}H^{\prime}(v_{n})\|_{(H^{s}(\mathbb{R}^{N}))^*}\rightarrow 0,
\end{equation}
where $H(u)=\frac{1}{2}\int_{\mathbb{R}^{N}}|u|^{2}dx$.

Since $h_{n}\rightarrow 0$ as $n\rightarrow\infty$ and
$$
\begin{aligned}
&\partial_h\tilde{F}(u_{n},h_{n})\\
=&\partial_h\Big(\frac{1}{2}e^{2sh_{n}}\|(-\Delta)^{s/2}u_{n}\|_{2}^{2}
-\frac{1}{p}\int_{\mathbb{R}^{N}}a(x)|e^{\frac{Nh_{n}}{2}}u_{n}(e^{h_{n}}x)|^{p}dx\Big)\\
=&s\|(-\Delta)^{s/2}e^{\frac{Nh_{n}}{2}}u_{n}(e^{h_{n}}x)\|_{2}^{2}-\frac{N(p-2)}{2p}\int_{\mathbb{R}^{N}}a(x)|e^{\frac{Nh_{n}}{2}}u_{n}(e^{h_{n}}x)|^{p}dx\\
&+\frac{1}{p}\int_{\mathbb{R}^{N}}(x\cdot\nabla a(x))|e^{\frac{Nh_{n}}{2}}u_{n}(e^{h_{n}}x)|^{p}dx,
\end{aligned}
$$
we obtain \eqref{p2}.

By $(A_{4})$, \eqref{p2} and \eqref{p6}, we have
$$
\begin{aligned}
&m_{a,c}+o(1) \\
=&\frac{Np-4s}{4ps}\int_{\mathbb{R}^{N}}a(x)|v_{n}|^{p}dx-\frac{1}{2ps}\int_{\mathbb{R}^{N}}(Na(x)+x\cdot\nabla a(x))|v_{n}|^{p}dx\\
\ge&\frac{Np-4s}{4ps}\int_{\mathbb{R}^{N}}a_{\ast}|v_{n}|^{p}dx-\frac{N}{2ps}\int_{\mathbb{R}^{N}}|v_{n}|^{p}dx\\
=&\frac{1}{2ps}\Big(\frac{Np-4s}{2}a_{\ast}-N\Big)\|v_{n}\|_{p}^{p}.
\end{aligned}
$$
Condition $(A_{1})$ implies that $\{\|v_{n}\|_{p}^{p}\}$ is bounded, and
\begin{equation}\label{p8}
\|v_{n}\|_{p}^{p}\le\frac{4psm_{a,c}}{(Np-4s)a_{\ast}-2N}+o(1).
\end{equation}
This together with $(A_{2})$, $(A_{3})$ and $(A_{4})$ yields that
\begin{equation}\label{p10}
\begin{aligned}
\|(-\Delta)^{s/2}v_{n}\|_{2}^{2}&=2m_{a,c}+\frac{2}{p}\int_{\mathbb{R}^{N}}a(x)|v_{n}|^{p}dx+o(1)\\
&\le2m_{a,c}+\frac{2}{p}\int_{\mathbb{R}^{N}}|v_{n}|^{p}dx+o(1)\\
&\le\Big(2+\frac{8s}{(Np-4s)a_{\ast}-2N}\Big)m_{a,c}+o(1),
\end{aligned}
\end{equation}
which implies that $\{v_n\}$ is bounded in $H^{s}(\mathbb{R}^{N})$. Moreover, we have
\begin{equation}\label{p11}
\int_{\mathbb{R}^{N}}x\cdot\nabla a(x)|v_{n}|^{p}dx\leq\|W\|_{\infty}\|v_{n}\|_{p}^{p}.
\end{equation}

In order to see \eqref{p4}, we set
$$
a_{n}:=\|(-\Delta)^{s/2}v_{n}\|_{2}^{2},\quad b_{n}:=\int_{\mathbb{R}^{N}}a(x)|v_{n}|^{p}dx,\quad
d_{n}:=\int_{\mathbb{R}^{N}}x\cdot\nabla a(x)|v_{n}|^{p}dx.
$$
Then from \eqref{p8}, \eqref{p10}, \eqref{p11} and the definition of $\lambda_n$, up to a subsequence, we can assume that
$$
a_{n}\rightarrow a\geq0,\quad b_{n}\rightarrow b\geq0,\quad d_{n}\rightarrow d\geq0,\quad \lambda_{n}\rightarrow \lambda\in\mathbb{R}.
$$
Passing to the limit in \eqref{p2}, \eqref{p6} and \eqref{p7}, then
\begin{equation}\label{p12}
a=\frac{2}{p}b+2m_{a,c},
\end{equation}
\begin{equation}\label{p13}
sa+(\frac{N}{p}-\frac{N}{2})b+\frac{1}{p}d=0,
\end{equation}
\begin{equation}\label{p14}
a=\lambda c^{2}+b.
\end{equation}
Now, we claim that $\lambda<0$. In fact, from the above three equalities, we have
$$
\begin{aligned}
-\lambda c^{2}&=b-a\\
&=\frac{p-2}{p}b-2m_{a,c}\\
&=\frac{p-2}{p}\cdot\frac{2ps}{N(p-2)-4s}(2m_{a,c}+\frac{1}{ps}d)-2m_{a,c}\\
&=\frac{p(2s-N)+2N}{N(p-2)-4s}2m_{a,c}+\frac{2(p-2)}{p(N(p-2)-4s)}d\\
&>\frac{p(2s-N)+2N}{N(p-2)-4s}2m_{c}+\frac{2(p-2)}{p(N(p-2)-4s)}d,\\
\end{aligned}
$$
where we have used Lemmas \ref{lm} and \ref{inf}, from $2+\frac{4s}{N}<p<2_{s}^{*}$, we have $\lambda<0$. On the other hand,
$$
\begin{aligned}
-\lambda c^{2}&=\frac{p(2s-N)+2N}{N(p-2)-4s}2m_{a,c}+\frac{2(p-2)}{p(N(p-2)-4s)}d\\
&<\frac{p(2s-N)+2N}{N(p-2)-4s}4m_{c}+\frac{(p-2)\|W\|_{\infty}}{N(p-2)-4s}\cdot\frac{16s}{(Np-4s)a_{\ast}-2N}m_{c}=\delta_{0}.
\end{aligned}
$$
Thus, we have \eqref{p4} holds. Moreover, recalling the definition of $\theta$, we obtain
\begin{equation}\label{p15}
\begin{aligned}
\frac{\delta_{0}}{m_{c}}=&\frac{4p(2s-N)+8N}{N(p-2)-4s}+\frac{16s(p-2)\|W\|_{\infty}}{(N(p-2)-4s)((Np-4s)a_{\ast}-2N)}\\
\le&\frac{12N-6p(N-2s)}{N(p-2)-4s}
=3\theta.
\end{aligned}
\end{equation}
\end{proof}

\section{The Proofs of Theorems \ref{th1} and \ref{th2}}\label{proof}
{\bf Proof of Theorem \ref{th1}:}
Since $\{v_{n}\}$ is bounded in $H^{s}(\mathbb{R}^{N})$, after passing to a subsequence it converges weakly in $H^{s}(\mathbb{R}^{N})$ to $v \in H^{s}(\mathbb{R}^{N})$. By Lemma \ref{psbound}, $v \geq 0$, and by weak convergence, $v$ is a weak solution of
$$
(-\Delta)^{s}v=\lambda v+a(x)|v|^{p-2}v,
$$
such that
$$
\|v\|_{2}\leq \liminf_{n\rightarrow0} \|v_{n} \|_{2}=c.
$$
It remains to prove that $\|v\|_{2}=c$.
Since
$$
\int_{\mathbb{R}^{N}}(-\Delta)^{s/2}v_{n}(-\Delta)^{s/2}\varphi dx-\lambda_{n}\int_{\mathbb{R}^{N}}v_{n}\varphi dx-\int_{\mathbb{R}^{N}}a(x)|v_{n}|^{p-2} v_{n}\varphi dx=o(1)\|\varphi\|
$$
for every $\varphi \in H^{s}(\mathbb{R}^{N})$, then $\{v_{n}\}$ is a Palais-Smale sequence for $F_{\lambda}$ at level  $m_{a,c}-\frac{\lambda}{2}c^{2}$. Therefore, by Lemma \ref{split}, we have
$$
v_{n}=v+\sum_{j=1}^{k} u^{j}(\cdot-y_{n}^{j})+o(1),
$$
with $u^{j}$ being solutions to
$$
(-\Delta)^{s}u^{j}=\lambda u^{j}+|u^{j}|^{p-2} u^{j}
$$
and $|y_{n}^{j}| \rightarrow \infty$. We note that, if $k=0$, then $v_{n} \rightarrow v$ strongly in  $H^{s}(\mathbb{R}^{N})$, hence $\|v\|_{2}=c$ and we are done, thus we can assume that $k \geq 1$, or equivalently $\rho:=\|v\|_{2}<c$.

First, we exclude the case $v=0$. In fact, if $v=0$ and $k=1$, we would have $u^{1}>0$ and  $\|u^{1}\|_{2}=c$, so that \eqref{s3} gives $m_{a,c}=m_{c}$, which is not possible due to Lemma \ref{inf}. On the other hand, if  $k \geq 2$, by Lemma \ref{split}, we have
$$
F_{\lambda}(v_{n})\to\sum_{j=1}^{k}F_{\infty,\lambda}(u^{j})=\sum_{j=1}^{k}F_{\infty}(u^{j})
-\sum_{j=1}^{k}\frac{\lambda}{2}\|u^{j}\|^2_{2},\quad
c^{2}=\sum_{j=1}^{k}\|u^{j}\|_{2}^{2}.
$$
Let $\alpha^2_{j}=\|u^{j}\|_{2}^{2}$, since $F_{\infty}(u^{j}) \geq m_{\alpha_{j}}$ and
\begin{equation}\label{proof1}
m_{\alpha}>m_{\beta},\quad\text {if $\alpha<\beta$},
\end{equation}
we have $2m_{c}\leq m_{a,c}$, which contradicts Lemma \ref{sup}.

Therefore, from now on we assume that $v \neq 0$. From $F(v_{n}) \rightarrow m_{a,c}$, we deduce that
$$
m_{a,c}-\frac{\lambda}{2}c^{2}=F(v)-\frac{\lambda}{2} \rho^{2}+\sum_{j=1}^{k} F_{\infty}(u^{j})-\frac{\lambda}{2}\sum_{j=1}^{k}\alpha_{j}^{2}
$$
Using $F_{\infty}(u^{j}) \geq m_{\alpha_{j}}$, \eqref{proof1} and
$$
c^{2}=\rho^{2}+\sum_{j=1}^{k}\alpha_{j}^{2},
$$
 we have
$$
2m_{c}>m_{a,c}\geq F(v)+\sum_{j=1}^{k} m_{\alpha_{j}} \geq F(v)+m_{\alpha} \geq F(v)+m_{c},
$$
where $\alpha=\max\limits_{j} \alpha_{j}$. Moreover, using the equation for $v$, it is easy to check that
$$
F_{\lambda}(v)=\max_{t>0}F_{\lambda}(tv).
$$
Let $\beta>0$ be such that
$$
\lambda_{\beta}=\lambda,
$$
according to \eqref{lambdac}. By Appendix in \cite{SBP}, $w_{\beta}$ satisfies the limit equation with multiplier $\lambda$ and $\beta\leq\alpha$. 
Using Lemma \ref{sup} and the above arguments, we have
$$
\begin{aligned}
2m_{c}&>m_{a,c} \geq m_{\alpha}+F(v)=m_{\alpha}+F_{\lambda}(v)+\frac{\lambda}{2}\rho^{2}\\
&=m_{\alpha}+\max_{t>0}F_{\lambda}(tv)+\frac{\lambda}{2}\rho^{2}\\
&\geq m_{\alpha}+\max_{t>0}F_{\infty,\lambda}(tv)+\frac{\lambda}{2}(c^{2}-\alpha^{2})\\
&\geq m_{\alpha}+F_{\infty,\lambda}(w_{\beta})+\frac{\lambda}{2}(c^{2}-\alpha^{2}) \\
&\geq m_{\alpha}+m_{\beta}+\frac{\lambda}{2}(c^{2}-\alpha^{2}-\beta^{2}).
\end{aligned}
$$
Since $m_{\beta}\geq m_{\alpha}>m_{c}$ and $\lambda<0$, we deduce that
$$
c^{2}-\alpha^{2}-\beta^{2}>0.
$$
Using Lemma \ref{psbound} and \eqref{mm}, we obtain
$$
\begin{aligned}
2m_{c}&>\frac{c^{\theta}}{\alpha^{\theta}}m_{c}+\frac{c^{\theta}}{\beta^{\theta}}m_{c}+\frac{\lambda}{2}(c^{2}-\alpha^{2}-\beta^{2})\\
&>\frac{c^{\theta}}{\alpha^{\theta}}m_{c}+\frac{c^{\theta}}{\beta^{\theta}}m_{c}-\frac{\lambda_{0}m_{c}}{2c^{2} }(c^{2}-\alpha^{2}-\beta^{2}),
\end{aligned}
$$
where $\lambda_{0}:=\frac{\delta_{0}}{m_{c}}.$ Hence, one gets
\begin{equation}\label{proof2}
2+\frac{\lambda_{0}}{2}>\frac{c^{\theta}}{\alpha^{\theta}}+\frac{c^{\theta}}{\beta^{\theta}}+\frac{\lambda_{0}}{2}(\frac{\alpha^{2}}{c^{2}}+\frac{\beta^{2}}{c^{2}}).
\end{equation}

Note that the following inequality holds for $A>\theta$ (see \cite{TBR})
$$
\min\Big\{x^{-\theta/2}+y^{-\theta/2}+A(x+y):x,y>0,x+y\leq 1\Big\}\geq\frac{3}{2}\theta+2,
$$
and by using $(A_{3})$ and \eqref{p15}, we have
$$
\lambda_{0}=\frac{\delta_{0}}{m_{c}}>\frac{4(p(2s-N)+2N)}{N(p-2)-4s}=2\theta,
$$
so \eqref{proof2} implies that
$$
2+\frac{\lambda_{0}}{2}>\frac{3\theta}{2}+2,
$$
which is a contradiction  to \eqref{p15}.\\

{\bf Proof of Theorem \ref{th2}:}
The strategy is similar to the proof of Theorem \ref{th1}. We observe that the results obtained in Section 3 excepting Lemma \ref{rh} and Lemma \ref{sup} still hold true without any changes in the proof.

{\bf First, we claim that if $(A_{6})$ holds, then  $m_{a,c}<2m_{c}$.}

From the definition of $m_{c}$, we have

\begin{equation*}
m_{c}=\frac{1}{2}\|(-\Delta)^{s/2}w_{c}\|_{2}^{2}-\frac{1}{p}\|w_{c}\|_{p}^{p}
\end{equation*}
and
\begin{equation*}
s\|(-\Delta)^{s/2}w_{c}\|_{2}^{2}-\frac{N(p-2)}{2p}\int_{\mathbb{R}^{N}}a(x)|w_{c}|^{p}dx=0,
\end{equation*}
which implies that
\begin{equation}\label{e1}
\|(-\Delta)^{s/2}w_{c}\|_{2}^{2}=\frac{2N(p-2)}{N(p-2)-4s}m_{c},\quad\|w_{c}\|_{p}^{p}=\frac{4sp}{N(p-2)-4s}.
\end{equation}
Note that 
\begin{equation}\label{e2}
\begin{aligned}
m_{a,c}&\leq\max_{h\in \mathbb{R}, y\in\mathbb{R}^{N}}\Big\{F_{\infty}(h\star w_{c})+\frac{1}{p}\int_{\mathbb{R}^{N}}(1-a(x))|(h \star w_{c})(x-y)|^{p}dx\Big\}\\
&\leq \max_{h\in\mathbb{R}}\Big\{F_{\infty}(h\star w_{c})+\frac{e^{2sh}}{p}\|1-a(x)\|_{t_{1}}\|w_{c}\|_{t_{2}p}^{p}\Big\} \\
&=\max_{h\in\mathbb{R}}\Big(e^{2sh}\Big(\frac{1}{2}\|(-\Delta)^{s/2}w_{c}\|_{2}^{2}+\frac{1}{p}\|1-a(x)\|_{t_{1}}\|w_{c}\|_{t_{2}p}^{p}\Big)-\frac{e^{\frac{p-2}{2}Nh}}{p}\|w_{c}\|_{p}^{p}\Big).
\end{aligned}
\end{equation}
Differentiating with respect to $h$, then there exists $h_{0}$ such that
\begin{equation*}
2se^{2sh_{0}}\Big(\frac{1}{2}\|(-\Delta)^{s/2}w_{c}\|_{2}^{2}+\frac{1}{p}\|1-a(x)\|_{t_{1}}\|w_{c}\|_{t_{2}p}^{p}\Big)
-\frac{N(p-2)}{2p}e^{\frac{p-2}{2}Nh_{0}}\|w_{c}\|_{p}^{p}=0.
\end{equation*}
Hence, we obtain that
\begin{equation}\label{e3}
e^{\frac{p-2}{2}Nh_{0}}=\left\{\frac{N(p-2)-4s}{N(p-2)m_{c}}\Big(\frac{1}{2}\|(-\Delta)^{s/2}w_{c}\|_{2}^{2}+\frac{1}{p}\|1-a(x)\|_{t_{1}}\|w_{c}\|_{t_{2}p}^{p}\Big)\right\}^{\frac{N(p-2)}{N(p-2)-4s}}.
\end{equation}
Combining $(A_{6})$ with \eqref{e1} and \eqref{e3}, we obtain
\begin{equation*}
e^{\frac{p-2}{2}Nh_{0}}<2.
\end{equation*}
From \eqref{e1} and \eqref{e2}, we have
\begin{equation*}
m_{a,c}\leq e^{\frac{p-2}{2}Nh_{0}}m_{c}.
\end{equation*}
This ends the proof.

{\bf Second, if $(A_{6})$ holds, then Lemma \ref{rh} still holds true.}

  Same as the proof of \eqref{r2}, we can fix $h_{1}<0<h_{2}$ such that
\begin{equation}\label{f1}
\sup_{(y,h)\in \mathbb{R}^{N}\times\{h_{1}, h_{2}\}}F(h\star w_{c}(\cdot-y))\leq m_{c}.
\end{equation}
For all $(y, h)\in\mathbb{R}^{N}\times\mathbb{R}$, we have
\begin{equation*}\label{f2}
\begin{aligned}
&\limsup_{|y|\to +\infty}\max_{h\in[h_{1}, h_{2}]}F(h\star w_{c}(x-y))\\
=&\limsup_{|y|\to +\infty}\max_{h\in[h_{1}, h_{2}]}\Big(F_{\infty}(h\star w_{c})+\frac{1}{p}\int_{\mathbb{R}^{N}}(1-a(x))|(h \star w_{c})(x-y)|^{p}dx\Big).
\end{aligned}
\end{equation*}
Since $w_{c}\in L^{t_{2}p}(\mathbb{R}^{N})$ and $(1-a(x))\in L^{t_{1}}(\mathbb{R}^{N})$, there exists $R>0$ large enough for any $\varepsilon>0$ such that 
\begin{equation*}
\|w_{c}\|_{L^{t_{2}p}(B_{R}(y))}^{p}<\varepsilon,\quad
\|1-a(x)\|_{L^{t_{1}}(\mathbb{R}^{N} \backslash B_{R}(0))}<\varepsilon, \quad \text{for}\ |y|\rightarrow\infty.
\end{equation*}
Therefore, 
\begin{equation*}\label{f3}
\begin{aligned}
&\max_{h\in[h_{1}, h_{2}]}\int_{\mathbb{R}^{N}}(1-a(x))|(h \star w_{c})(x-y)|^{p}dx\\
=&\max_{h\in[h_{1}, h_{2}]}\left(\int_{B_{R}(0)}(1-a(x))|(h \star w_{c})(x-y)|^{p}dx+\int_{\mathbb{R}^{N}\backslash B_{R}(0)}(1-a(x))|(h \star w_{c})(x-y)|^{p}dx\right)\\
\leq&\|1-a(x)\|_{t_{1}}\max_{h\in[h_{1}, h_{2}]}e^{2sh}\|w_{c}\|_{L^{t_{2}p}(B_{R}(y))}^{p}+\|1-a(x)\|_{L^{t_{1}}(\mathbb{R}^{N} \backslash B_{R}(0))}\max_{h \in[h_{1}, h_{2}]}\|h\star w_{c}\|_{t_{2}p}^{p}\rightarrow0.\\
\end{aligned}
\end{equation*}
Using $(A_{3})$, $(A_{4})$ and $(A_{6})$,
$$
\limsup_{|y|\to +\infty}\max_{h\in[h_{1}, h_{2}]}F(h\star w_{c}(x-y))\leq m_{c}.
$$
Moreover,
$$
\limsup_{|y|\to +\infty}F(w_{c}(x-y))\ge \limsup_{|y|\to +\infty}F_\infty(w_{c}(x-y))=F_\infty(w_{c})=m_c,
$$
which implies that
\begin{equation}\label{f4}
\limsup_{|y|\to +\infty}\max_{h\in[h_{1},h_{2}]}F(h\star w_{c}(x-y))= m_{c}.
\end{equation}
This together with \eqref{f1} and \eqref{f4} yields that Lemma \ref{rh} holds.

\end{document}